\documentclass[12pt,a4paper,oneside,onecolumn,fleqn]{article}
\usepackage{mathrsfs}
\usepackage{amsfonts}
\usepackage{longtable}
\usepackage{mathtools}
\usepackage{colortbl}
\usepackage{latexsym}
\usepackage{amsmath,amsthm}
\usepackage{epsf}
\usepackage{graphicx}
\usepackage{indentfirst}
\usepackage{cite}
\usepackage[numbers,sort&compress]{natbib}
\usepackage{caption2}

\renewcommand{\arraystretch}{1}

\theoremstyle{plain}
\newtheorem{thm}{\bf Theorem}[section]
\newtheorem{con}[thm]{\bf Construction}

\newtheorem{lem}[thm]{\bf Lemma}

\theoremstyle{definition}

\theoremstyle{remark}

\title{\bf A generalization of product construction of multimagic squares
}
\author
 {   {  \small  Yong Zhang$^{1,3}$,\ \ Kejun Chen$^{2}$, \ \ Wen Li$^{1}$ } \\ %\footnote{ Corresponding to:  K.Chen(yctuckj@163.com)}
            {\footnotesize \it  1. School of Mathematical Sciences, Yancheng Teachers University, Jiangsu 224002, China}\\
            \footnotesize { \it 2.  School of Mathematics and Physics, Taizhou University, Jiangsu 225300,   China}\\
             \footnotesize { \it 3. Department of Computer Science, Lakehead University, Thunder Bay, ON P7B 5E1}\\
 }
\date{}
\begin{document}

\maketitle
% abstract and keywords
{\baselineskip 15pt
\begin {abstract}
\noindent
   In this paper, constructions of multimagic squares are investigated. Diagonal Latin squares and Kronecker products are used to get some constructions  of multimagic squares. Consequently,  some new families of compound multimagic squares are obtained.
\end{abstract}}

 {\begingroup\makeatletter\let\@makefnmark\relax\footnotetext{Supported by the Natural Science Foundations of China ( No.11301457, No.11371308).}

\vskip 0.5cm
\section{Introduction}

 An $n\times n$ matrix $A$ consisting of $n^2$   integers is a {\it general magic square of order $n$} if the sum of $n$ elements in each row, each column and each of two   diagonals are the same.  The sum is the \emph{magic number}.   A general magic square of order $n$ is  a \emph{magic square}, denoted by  MS$(n)$,  if the entries  are $n^2$ consecutive integers.  Usually, the $(i,j)$ entry of a matrix $A$ is denoted by $a_{i,j}$.    A lot of work has been done on   magic squares (\cite{Abe,Ahmed,Andrews,Cammann,Handbook,Nordgren}).

Let $t$ be a positive  integer.  A   general magic square $M$   is a \emph{general $t$-multimagic square}  if  $M^{*1}$,$M^{*2}$, $\cdots$, $M^{*t}$ are all general magic squares, where  $M^{*e}=(m_{i,j}^{e})$, $e=1,2,\dots,t$. A  general $t$-multimagic square of order $n$ is  a \emph{$t$-multimagic square}, denoted by  MS($n,t$), if the entries  are $n^2$ consecutive integers.
Usually, a $2$-multimagic square is called a \emph{bimagic square} and a $3$-multimagic square   is called a \emph{trimagic squares}.

In 2007, Derksen, Eggermont and  van den Essen  \cite{Derksen} used the matrices over rings to provide a constructive procedure.
 Recently, Chen and Li \cite{Chen2} introduced a magic pair of Latin squares to get a construction of bimagic squares.    Zhang, Lei and Chen \cite{Zhang,Zhang2} used large sets of orthogonal arrays to get a construction  of multimagic squares.

 In this paper,  some new constructions of multimagic squares are provided by using diagonal Latin squares with Kronecker products.

A {\it Latin square} of order $m$, denoted by LS$(n)$, is an $m\times m$ array such that every row and every column is a permutation of an $n$-set $S$.    A Latin square   over $S$ is called {\it diagonal}  if it has the   property that the main diagonal and back diagonal are both permutations of $S$.
  Two LS$(n)$s are called {\it orthogonal} if each symbol in the first square meets each symbol in the second square exactly once when they are superposed.

 Let $I_n=\{0,1,\cdots, n-1\}$ and  let  $J_{m\times n}$ be a matrix with the  entries being  all 1s.
For an MS$(n,t)$ $A$, if  the smallest entry integer
is $s$, then it is readily verified that   $A-sJ_{n\times n}$ is  an MS$(n,t)$   over $I_{n^2}$.  In the sequel, the element set of  an MS$(n,t)$ is always taken to be $I_{n^2}$.
  For an $m\times n$ matrix $A$, the rows and columns of $A$ are indexed by $I_m$ and $I_n$, respectively.

Given an $m\times n$ matrix $A$ and a $p\times q$ matrix $B$, the \emph{Kronecker product} $A\otimes B$ is the $mp\times nq$ matrix given by
$A\otimes B=(a_{i,j}B)$, where $a_{i,j}B=(a_{i,j}b_{s,t})$, $i\in I_m$, $j\in I_n$, $s\in I_p$, $t\in I_q$.
For an MS$(m)$  $A$    and an MS$(n)$  $B$, we define a \emph{product} of $A, B$ as follows.
\begin{eqnarray}
\nonumber\mbox{}\hspace{1.8in}   A*B=n^2 A\otimes J_{n\times n}+ J_{m\times m}\otimes B,
\end{eqnarray}
i.e.,
\begin{center}
$A*B={\footnotesize\left(\begin{array}{ cccc}
n^2a_{0,0}J_{n\times n}+B&n^2a_{0,1}J_{n\times n}+B&\cdots&n^2a_{0,m-1}J_{n\times n}+B\\
n^2a_{1,0}J_{n\times n}+B&n^2a_{1,1}J_{n\times n}+B&\cdots&n^2a_{1,m-1}J_{n\times n}+B\\
\vdots&\vdots&\vdots&\vdots\\
n^2a_{m-1,0}J_{n\times n}+B&n^2a_{m-1,1}J_{n\times n}+B&\cdots&n^2a_{m-1,m-1}J_{n\times n}+B
\end{array}
 \right).\ \ \ \ \ }$
 \end{center}
 It was proved in \cite{Kim} that $A*B$ is an  MS$(mn)$. Such a magic square is also  called a \emph{compound} magic square or a \emph{composite} magic square  in earlier papers such as  \cite{Andrews,Derksen} and recent papers such as \cite{Li}.   Further,    Derksen et al \cite{Derksen}  proved that
if $A$ is an  MS$(m,t)$ and $B$ is an  MS($n,t$), then the product $A*B$ is an MS$(mn,t)$.

The main purpose of this paper is to generalize the product construction to get more composite multimagic squares. Some new constructions of multimagic squares are provided in Section 2, and some new families of  bimagic squares and trimagic squares are  given in  Section 3 and Section 4, respectively.

\section{Constructions of multimagic squares}

Constructions of multimagic squares are investigated in this section.
In the product construction of multimagic squares, one can find that $J_{m\times m}\otimes B$    is a general $t$-multimagic square of order $mn$  whenever  $B$  is an MS$(n,t)$. We shall prove that the conclusion is still true if $J_{m\times m}\otimes B$ is replaced by a general $t$-multimagic square $\mathcal{B}$ as follows:
\begin{center}
$\mathcal{B}={\footnotesize\left(\begin{array}{ cccc}
B_{0,0}&B_{0,1}&\cdots&B_{0,m-1}\\
B_{1,0}&B_{1,1}&\cdots&B_{1,m-1}\\
\vdots&\vdots&\vdots&\vdots\\
B_{m-1,0}&B_{m-1,1}&\cdots&B_{m-1,m-1}
\end{array}
 \right),}$\\
 \end{center}
\noindent where $B_{u,v}$ is an MS$(n,t-1)$, $u,v\in I_m$. Such  a partitioned general $t$-multimagic square $\mathcal{B}$ is  denoted by PGMS$(mn,t)$.  We have the following.

 \begin{con}\label{newthm}
If there is an  MS$(m,t)$    and
 there is a PGMS$(mn,t)$,
    then there is an MS$(mn,t)$.
\end{con}
\begin{proof}
 Let  $A$  be an  MS$(m,t)$ and let  $\mathcal{B}$ be a  PGMS$(mn,t)$,  $\mathcal{B}=(B_{u,v})$, $B_{u,v}=(b^{(u,v)}_{r,s}), u,v\in I_m, r,s\in I_n$,  and let
 \begin{center}
  $ C=n^2 A\otimes J_{n\times n}+ \mathcal{B}$,
   \end{center}
i.e.,
$$c_{i,j}=n^2a_{u,v}+b^{(u,v)}_{r,s}, i=nu+r, j=nv+s, u,v\in I_m,\  r,s\in I_n.$$
We shall prove that $C$ is an MS$(mn,t)$.

 Since $\{a_{u,v}|u,v\in I_m\}=I_{m^2}$ and $\{b^{(u,v)}_{r,s}|r,s\in I_n\}=I_{n^2}, u,v\in I_m$, we have
 $$\{c_{i,j}|i,j\in I_{mn}\}=\{n^2a_{u,v}+b^{(u,v)}_{r,s}|u,v\in I_m,\  r,s\in I_n\}=I_{(mn)^2},$$
 which means that  the $(mn)^2$ entries of $C$ are $0,1,\cdots, (mn)^2-1$.

For an MS$(n,t)$ $A$, the magic number of $A^{*e}$ is $S_e(n)$, where $$S_e(n)=\frac{\sum\limits_{k\in I_{n^2}}k^e}{n}, e=1,2,\cdots,t.$$ It remains  to prove that for each $e=2, 3, \cdots,t$,
\begin{center}
$\sum\limits_{j=0}^{mn-1}c_{i,j}^e=S_e(mn), i\in I_{mn}; \sum\limits_{i=0}^{mn-1}c_{i,j}^e=S_e(mn), j\in I_{mn};$\\
\ \ \ \ $\sum\limits_{i=0}^{mn-1}c_{i,i}^e=S_e(mn); \sum\limits_{i=0}^{mn-1}c_{i,mn-1-i}^e=S_e(mn). \ \ \ \ \  \  \ \ \ \ \ \ \ \  \  \ \ \ $
\end{center}

In fact, by the definition of $C$,  for each $e=2, 3, \cdots,t$, we have
\begin{center}
$  \sum\limits_{j=0}^{mn-1}c_{i,j}^e=\sum\limits_{k=0}^{e}\binom{e}{k} n^{2(e-k)} \sum\limits_{v=0}^{m-1}a_{u,v}^{e-k}\sum\limits_{s=0}^{n-1}(b^{(u,v)}_{r,s})^k.$
\end{center}
 By hypothesis,  $ \sum\limits_{v=0}^{m-1}a_{u,v}^{e}=S_{e}(m), u\in I_m, e=2, 3, \cdots,t$,  and $\sum\limits_{s=0}^{n-1}(b^{(u,v)}_{r,s})^e=S_{e}(n), r\in I_n, e=2,3,\cdots,t-1$.

 If $e<t$, we have
\begin{center}
 $ \sum\limits_{j=0}^{mn-1}c_{i,j}^e= \sum\limits_{k=0}^{e}\binom{e}{k} n^{2(e-k)} S_{e-k}(m)S_{k}(n).$
\end{center}
If   $e=t$, since  $B$ is a general $t$-multimagic square,  $\sum\limits_{v=0}^{m-1}\sum\limits_{s=0}^{n-1}(b^{(u,v)}_{r,s})^t=mS_{t}(n),  s\in I_n$,  we have
\begin{center}
$  \sum\limits_{j=0}^{mn-1}c_{i,j}^e=\sum\limits_{k=0}^{t-1}\binom{t}{k} n^{2(t-k)} \sum\limits_{v=0}^{m-1}a_{u,v}^{t-k}\sum\limits_{s=0}^{n-1}(b^{(u,v)}_{r,s})^k+
\sum\limits_{v=0}^{m-1}\sum\limits_{s=0}^{n-1}(b^{(u,v)}_{r,s})^t $\\
$=\sum\limits_{k=0}^{t-1}\binom{t}{k} n^{2(t-k)} S_{t-k}(m)S_{k}(n)+ mS_{t}(n),\ \        \  \ \  $\\
$=\sum\limits_{k=0}^{t}\binom{t}{k} n^{2(t-k)} S_{t-k}(m)S_{k}(n).\ \ \ \ \  \ \ \ \ \    \ \ \ \ \ \ \  \ \ $
\end{center}
Therefore for each    $e(2\leq e\leq t)$,  we have
\begin{center}
 $\sum\limits_{j=0}^{mn-1}c_{i,j}^e=\sum\limits_{k=0}^{e}\binom{e}{k} n^{2(e-k)} S_{e-k}(m)S_{k}(n).$
\end{center}
The right hand of the above equality only depends on $m, n, e$, denoted by N$(m, n, e)$. Therefore
\begin{center}
$\sum\limits_{i=0}^{mn-1}\sum\limits_{j=0}^{mn-1}c_{i,j}^e=mnN(m, n, e).$
 \end{center}
 On the other hand, we have
 \begin{center}
  $\sum\limits_{i=0}^{mn-1}\sum\limits_{j=0}^{mn-1}c_{i,j}^e=\sum\limits_{d=0}^{(mn)^2-1}d^e=mnS_e(mn).$
   \end{center}
which implies  $N(m, n, e)=S_e(mn)$, thus $\sum\limits_{j=0}^{mn-1}c_{i,j}^e=S_e(mn)$.

Similarly, one can prove that
\begin{center}
$\sum\limits_{i=0}^{mn-1}c_{i,j}^e=S_e(mn), j\in I_{mn}, $
$\sum\limits_{i=0}^{mn-1}c_{i,i}^e=\sum\limits_{i=0}^{mn-1}c_{i,mn-1-i}^e=S_e(mn).$
\end{center}
Therefore, $C$ is an MS$(mn,t)$.
\end{proof}

   In the following, we shall showed that a PGMS$(mn,t)$ can be obtained by using a set of complementary $t$-multimagic squares and a diagonal Latin square.

Let   $B_0, B_1, \cdots, B_{m-1}$ be $m$ MS($n,t$)s, where $B_s=(b_{i,j}^{(s)})$, $i,j\in I_n$, $s\in I_m$. They  are \emph{complementary}   if the following four conditions are satisfied.
\vskip 5pt

 \noindent  $(R1)$ \ \ \ \ \ \ \ \ \ \ \ \ \ \  \   $\sum\limits_{s\in I_m}\sum\limits_{j\in I_n}(b_{i,j}^{(s)})^{t+1}=mS_{t+1}(n),   i\in I_n$; \vskip 5pt

 \noindent  $(R2)$ \ \ \ \ \ \ \ \ \ \ \ \ \ \  \   $\sum\limits_{s\in I_m}\sum\limits_{i\in I_n}(b_{i,j}^{(s)})^{t+1}=mS_{t+1}(n),   j\in I_n$; \vskip 5pt

 \noindent  $(R3)$ \ \ \ \ \ \ \ \ \ \ \ \ \ \  \  $\sum\limits_{s\in I_m}\sum\limits_{i\in I_n}(b_{i,i}^{(s)})^{t+1}=mS_{t+1}(n)$;

  \noindent  $(R4)$ \ \ \ \ \ \ \ \ \ \ \ \ \ \  \  $\sum\limits_{s\in I_m}\sum\limits_{i\in I_n}(b_{i,n-1-i}^{(s)})^{t+1}=mS_{t+1}(n)$.

 \noindent  The set of  $B_0, B_1, \cdots, B_{m-1}$ is denoted by $m$-SCMS($n,t$). Usually,   $t$  omitted if $t=1$.

   \vskip10pt
As an example,   a set of  2 complementary magic squares is listed below.

 \vskip10pt
\noindent {\bf Example 1} \ \
Let
{\renewcommand\arraystretch{0.7}
\setlength{\arraycolsep}{1.5pt}
\begin{center}
     $B_0=${\scriptsize$\left(
        \begin{array}{cccc}
2&12&5&11\\
9&7&14&0\\
15&1&8&6\\
4&10&3&13
        \end{array}
        \right),$}
             $B_1=${\scriptsize$\left(
        \begin{array}{cccc}
1&15&6&8\\
10&4&13&3\\
12&2&11&5\\
7&9&0&14
        \end{array}
        \right).  $}
\end{center}}
\noindent  Then  $\{B_0, B_1\}$ is a $2$-SCMS($4$). In fact, it is easily checked that $B_0, B_1$ are MS$(4)$s and\\
$\mbox{}\hspace{1.5in}
 \sum\limits_{j\in I_4}(b_{i,j}^{(0)})^2+\sum\limits_{j\in I_4}(b_{i,j}^{(1)})^2=2S_2(4)=620,   i\in I_4$,\\
$\mbox{}\hspace{1.5in}  \sum\limits_{i\in I_4} (b_{i,j}^{(0)})^2+ \sum\limits_{i\in I_4} (b_{i,j}^{(1)})^2=620,   j\in I_4,$\\
$\mbox{}\hspace{1.5in}\sum\limits_{i\in I_4}(b_{i,i}^{(0)})^2+\sum\limits_{i\in I_4}(b_{i,i}^{(1)})^2=620$,\\
$\mbox{}\hspace{1.5in}\sum\limits_{i\in I_4}(b_{i,3-i}^{(0)})^2+\sum\limits_{i\in I_4}(b_{i,3-i}^{(1)})^2=620$.
\qed

\begin{con} \label{conmscms}
Let $t\geq2$.   If there is an  MS$(m,t)$ and there is an $m$-SCMS$(n,t-1)$,  then there is an MS$(mn,t)$.
\end{con}
 \begin{proof} By hypothesis we know that $m\geq4$ since there is no MS$(3,2)$ (see \cite{Boyer}).
By Reference \cite{Gergely} we know that a diagonal  LS$(m)$ exists if and only if $m\geq4$.  Let $D$ be a diagonal LS$(m)$ over $I_m$.
   Let $\{B_{0},B_{1},\cdots, B_{m-1}\}$ be an $m$-SCMS$(n,t-1)$, where $B_s=(b_{x,y}^{(s)})$, $x,y\in I_n, s\in I_m$. and let $P_m(u,v)$ be a matrix of order $m$ with the $(u,v)$ entry 1 and all other entries 0s.  Define a matrix $\mathcal{B}$ of order $mn$ as follows.
\begin{center}
$\mathcal{B}=\sum\limits_{u,v\in I_m}P_m(u,v)\otimes B_{d_{u,v}}.$
\end{center}
 We shall show that $\mathcal{B}$  is an PGMS$(mn,t)$.

 In fact,  if we denote $\mathcal{B}=(b_{i,j})_{mn\times mn}, i,j\in I_{mn}$,  then
\begin{center}
$b_{i,j}=b^{(d_{u,v})}_{x,y}, i=nu+x, j=nv+y, u,v\in I_m,\  x,y\in I_n$.
\end{center}
Since $D$ is  a diagonal  LS$(m)$, we have $\{d_{u,v}|v\in I_m\}=I_m, u\in I_m$. For each $i\in I_{mn}$ and for each $e\in \{1,2,\cdots, t\}$,
\begin{center}$
  \sum\limits_{j=0}^{mn-1}(b_{i,j})^{e}=\sum\limits_{v=0}^{m-1} \sum\limits_{y=0}^{n-1} (b^{(d_{u,v})}_{x,y})^{e}
=\sum\limits_{v=0}^{m-1} \sum\limits_{y=0}^{n-1} (b^{(v)}_{x,y})^{e}.
%=mS_{e}(n).
 $\end{center}
 If $e<t$, then $\sum\limits_{v=0}^{m-1} \sum\limits_{y=0}^{n-1} (b^{(v)}_{x,y})^{e}=m\sum\limits_{y=0}^{n-1} (b^{(v)}_{x,y})^{e}=mS_{e}(n)$ since $B_0, B_1, \cdots, B_{m-1}$ are MS$(n,t-1)$s.   If $e=t$, then $\sum\limits_{v=0}^{m-1} \sum\limits_{y=0}^{n-1} (b^{(v)}_{x,y})^{t}=mS_{t}(n)$ since $B_0, B_1, \cdots, B_{m-1}$ are complementary. Thus  $\sum\limits_{j=0}^{mn-1}(b_{i,j})^{e}=mS_{e}(n)$ for each $e\in \{1,2,\cdots, t\}$.

Similarly, for each $e\in \{1,2,\cdots, t\}$ we have
\begin{center}
$\sum\limits_{i=0}^{mn-1}(b_{i,j})^{e}=mS_{t}(n), j\in I_{mn}; $\\
$ \sum\limits_{i=0}^{mn-1}(b_{i,i})^{e}=mS_{e}(n)$, \ \ \ \ \ \ \ \ \ \ \ \ \\
  $\sum\limits_{i=0}^{mn-1}(b_{i,mn-1-i})^{e}=mS_{e}(n). $ \ \ \ \
 \end{center}
 So, $\mathcal{B}$ is a PGMS$(mn,t)$.  By   Construction   \ref{newthm} the proof is completed.
 \end{proof}

We use the following example   to illustrate the proofs of  Construction \ref{conmscms}.

\vskip5pt
\noindent {\bf Example 2}
  Let $B_i=B_0, B_{i+1}=B_1, i=2,4,6$, where $B_0, B_1$ are given in Example 1. It is readily checked that $\{B_0, B_1, \cdots, B_7\}$ is an $8$-SCMS($4$).
A DLS$(8)$ over $I_8$ is listed below.

{\renewcommand\arraystretch{0.6}
\setlength{\arraycolsep}{2.5pt}
\begin{center}
$D=${\scriptsize$\left(
        \begin{array}{cccccccc}
0	&3	&6	&5	&4	&7	&2	&1\\
1	&2	&7	&4	&5	&6	&3	&0\\
5	&6	&3	&0	&1	&2	&7	&4\\
4	&7	&2	&1	&0	&3	&6	&5\\
2	&1	&4	&7	&6	&5	&0	&3\\
3	&0	&5	&6	&7	&4	&1	&2\\
7	&4	&1	&2	&3	&0	&5	&6\\
6	&5	&0	&3	&2	&1	&4	&7
 \end{array}
\right).$}
\end{center}}
\noindent
 Let $\mathcal{B}=\sum\limits_{u,v\in I_8}P_8(u,v)\otimes B_{d_{u,v}}$, i.e.,
{\renewcommand\arraystretch{0.6}
\setlength{\arraycolsep}{0.6pt}
\begin{center}
$\mathcal{B}=${\tiny$\left(
        \begin{array}{cccc|cccc|cccc|cccc|cccc|cccc|cccc|cccc}
 2&12&5&11&1&15&6&8&2&12&5&11&1&15&6&8&2&12&5&11&1&15&6&8&2&12&5&11&1&15&6&8\\
9&7&14&0&10&4&13&3&9&7&14&0&10&4&13&3&9&7&14&0&10&4&13&3&9&7&14&0&10&4&13&3\\
15&1&8&6&12&2&11&5&15&1&8&6&12&2&11&5&15&1&8&6&12&2&11&5&15&1&8&6&12&2&11&5\\
4&10&3&13&7&9&0&14&4&10&3&13&7&9&0&14&4&10&3&13&7&9&0&14&4&10&3&13&7&9&0&14\\\hline
1&15&6&8&2&12&5&11&1&15&6&8&2&12&5&11&1&15&6&8&2&12&5&11&1&15&6&8&2&12&5&11\\
10&4&13&3&9&7&14&0&10&4&13&3&9&7&14&0&10&4&13&3&9&7&14&0&10&4&13&3&9&7&14&0\\
12&2&11&5&15&1&8&6&12&2&11&5&15&1&8&6&12&2&11&5&15&1&8&6&12&2&11&5&15&1&8&6\\
7&9&0&14&4&10&3&13&7&9&0&14&4&10&3&13&7&9&0&14&4&10&3&13&7&9&0&14&4&10&3&13\\\hline
1&15&6&8&2&12&5&11&1&15&6&8&2&12&5&11&1&15&6&8&2&12&5&11&1&15&6&8&2&12&5&11\\
10&4&13&3&9&7&14&0&10&4&13&3&9&7&14&0&10&4&13&3&9&7&14&0&10&4&13&3&9&7&14&0\\
12&2&11&5&15&1&8&6&12&2&11&5&15&1&8&6&12&2&11&5&15&1&8&6&12&2&11&5&15&1&8&6\\
7&9&0&14&4&10&3&13&7&9&0&14&4&10&3&13&7&9&0&14&4&10&3&13&7&9&0&14&4&10&3&13\\\hline
2&12&5&11&1&15&6&8&2&12&5&11&1&15&6&8&2&12&5&11&1&15&6&8&2&12&5&11&1&15&6&8\\
9&7&14&0&10&4&13&3&9&7&14&0&10&4&13&3&9&7&14&0&10&4&13&3&9&7&14&0&10&4&13&3\\
15&1&8&6&12&2&11&5&15&1&8&6&12&2&11&5&15&1&8&6&12&2&11&5&15&1&8&6&12&2&11&5\\
4&10&3&13&7&9&0&14&4&10&3&13&7&9&0&14&4&10&3&13&7&9&0&14&4&10&3&13&7&9&0&14\\\hline
2&12&5&11&1&15&6&8&2&12&5&11&1&15&6&8&2&12&5&11&1&15&6&8&2&12&5&11&1&15&6&8\\
9&7&14&0&10&4&13&3&9&7&14&0&10&4&13&3&9&7&14&0&10&4&13&3&9&7&14&0&10&4&13&3\\
15&1&8&6&12&2&11&5&15&1&8&6&12&2&11&5&15&1&8&6&12&2&11&5&15&1&8&6&12&2&11&5\\
4&10&3&13&7&9&0&14&4&10&3&13&7&9&0&14&4&10&3&13&7&9&0&14&4&10&3&13&7&9&0&14\\\hline
1&15&6&8&2&12&5&11&1&15&6&8&2&12&5&11&1&15&6&8&2&12&5&11&1&15&6&8&2&12&5&11\\
10&4&13&3&9&7&14&0&10&4&13&3&9&7&14&0&10&4&13&3&9&7&14&0&10&4&13&3&9&7&14&0\\
12&2&11&5&15&1&8&6&12&2&11&5&15&1&8&6&12&2&11&5&15&1&8&6&12&2&11&5&15&1&8&6\\
7&9&0&14&4&10&3&13&7&9&0&14&4&10&3&13&7&9&0&14&4&10&3&13&7&9&0&14&4&10&3&13\\\hline
1&15&6&8&2&12&5&11&1&15&6&8&2&12&5&11&1&15&6&8&2&12&5&11&1&15&6&8&2&12&5&11\\
10&4&13&3&9&7&14&0&10&4&13&3&9&7&14&0&10&4&13&3&9&7&14&0&10&4&13&3&9&7&14&0\\
12&2&11&5&15&1&8&6&12&2&11&5&15&1&8&6&12&2&11&5&15&1&8&6&12&2&11&5&15&1&8&6\\
7&9&0&14&4&10&3&13&7&9&0&14&4&10&3&13&7&9&0&14&4&10&3&13&7&9&0&14&4&10&3&13\\\hline
2&12&5&11&1&15&6&8&2&12&5&11&1&15&6&8&2&12&5&11&1&15&6&8&2&12&5&11&1&15&6&8\\
9&7&14&0&10&4&13&3&9&7&14&0&10&4&13&3&9&7&14&0&10&4&13&3&9&7&14&0&10&4&13&3\\
15&1&8&6&12&2&11&5&15&1&8&6&12&2&11&5&15&1&8&6&12&2&11&5&15&1&8&6&12&2&11&5\\
4&10&3&13&7&9&0&14&4&10&3&13&7&9&0&14&4&10&3&13&7&9&0&14&4&10&3&13&7&9&0&14\\
 \end{array}
\right).$}
\end{center}
One can check that $\mathcal{B}$ is  a PGMS$(32,2)$.

Let $A=P_8-J_{8\times 8}$, where $P_8$ is the   bimagic square of order 8 given by Pfeffermann \cite{PF} in 1891. Then
  \begin{center}
 {\renewcommand\arraystretch{0.6}
\setlength{\arraycolsep}{2.8pt}
$A=$\scriptsize
    $\left(
     \begin{array}{ccccccccccccc }
 55&	33	&7	&56	&17	&46&	8	&30\\
32	&19	&53	&47	&6	&28&	58	&9\\
25&	42	&12&	22	&63&	37	&3	&48\\
18	&4&	34	&29&	52&	11	&45	&59\\
14	&24	&62	&1	&40	&23	&49	&39\\
5	&54	&16	&10&	35&	57	&31	&44\\
60	&15	&41	&51	&26	&0	&38&	21\\
43	&61	&27&36	&13	&50&	20	&2\\
\end{array}
        \right)   $}
 \end{center}
 is a bimagic square over $I_{64}$.
  Let
 $ C=16 A \otimes J_{4\times 4}+\mathcal{B}$, i.e.,
{\renewcommand\arraystretch{0.6}
\setlength{\arraycolsep}{0.6pt}
\begin{center}
$C=${\tiny$\left(
        \begin{array}{cccc|cccc|cccc|cccc|cccc|cccc|cccc|cccc}
882&892&885&891&529&543&534&536&114&124&117&123&897&911&902&904&274&284&277&283&737&751&742&744&130&140&133&139&481&495&486&488\\
889&887&894&880&538&532&541&531&121&119&126&112&906&900&909&899&281&279&286&272&746&740&749&739&137&135&142&128&490&484&493&483\\
895&881&888&886&540&530&539&533&127&113&120&118&908&898&907&901&287&273&280&278&748&738&747&741&143&129&136&134&492&482&491&485\\
884&890&883&893&535&537&528&542&116&122&115&125&903&905&896&910&276&282&275&285&743&745&736&750&132&138&131&141&487&489&480&494\\\hline
513&527&518&520&306&316&309&315&849&863&854&856&754&764&757&763&97&111&102&104&450&460&453&459&929&943&934&936&146&156&149&155\\
522&516&525&515&313&311&318&304&858&852&861&851&761&759&766&752&106&100&109&99&457&455&462&448&938&932&941&931&153&151&158&144\\
524&514&523&517&319&305&312&310&860&850&859&853&767&753&760&758&108&98&107&101&463&449&456&454&940&930&939&933&159&145&152&150\\
519&521&512&526&308&314&307&317&855&857&848&862&756&762&755&765&103&105&96&110&452&458&451&461&935&937&928&942&148&154&147&157\\\hline
401&415&406&408&674&684&677&683&193&207&198&200&354&364&357&363&1009&1023&1014&1016&594&604&597&603&49&63&54&56&770&780&773&779\\
410&404&413&403&681&679&686&672&202&196&205&195&361&359&366&352&1018&1012&1021&1011&601&599&606&592&58&52&61&51&777&775&782&768\\
412&402&411&405&687&673&680&678&204&194&203&197&367&353&360&358&1020&1010&1019&1013&607&593&600&598&60&50&59&53&783&769&776&774\\
407&409&400&414&676&682&675&685&199&201&192&206&356&362&355&365&1015&1017&1008&1022&596&602&595&605&55&57&48&62&772&778&771&781\\\hline
290&300&293&299&65&79&70&72&546&556&549&555&465&479&470&472&834&844&837&843&177&191&182&184&722&732&725&731&945&959&950&952\\
297&295&302&288&74&68&77&67&553&551&558&544&474&468&477&467&841&839&846&832&186&180&189&179&729&727&734&720&954&948&957&947\\
303&289&296&294&76&66&75&69&559&545&552&550&476&466&475&469&847&833&840&838&188&178&187&181&735&721&728&726&956&946&955&949\\
292&298&291&301&71&73&64&78&548&554&547&557&471&473&464&478&836&842&835&845&183&185&176&190&724&730&723&733&951&953&944&958\\\hline
226&236&229&235&385&399&390&392&994&1004&997&1003&17&31&22&24&642&652&645&651&369&383&374&376&786&796&789&795&625&639&630&632\\
233&231&238&224&394&388&397&387&1001&999&1006&992&26&20&29&19&649&647&654&640&378&372&381&371&793&791&798&784&634&628&637&627\\
239&225&232&230&396&386&395&389&1007&993&1000&998&28&18&27&21&655&641&648&646&380&370&379&373&799&785&792&790&636&626&635&629\\
228&234&227&237&391&393&384&398&996&1002&995&1005&23&25&16&30&644&650&643&653&375&377&368&382&788&794&787&797&631&633&624&638\\\hline
81&95&86&88&866&876&869&875&257&271&262&264&162&172&165&171&561&575&566&568&914&924&917&923&497&511&502&504&706&716&709&715\\
90&84&93&83&873&871&878&864&266&260&269&259&169&167&174&160&570&564&573&563&921&919&926&912&506&500&509&499&713&711&718&704\\
92&82&91&85&879&865&872&870&268&258&267&261&175&161&168&166&572&562&571&565&927&913&920&918&508&498&507&501&719&705&712&710\\
87&89&80&94&868&874&867&877&263&265&256&270&164&170&163&173&567&569&560&574&916&922&915&925&503&505&496&510&708&714&707&717\\\hline
961&975&966&968&242&252&245&251&657&671&662&664&818&828&821&827&417&431&422&424&2&12&5&11&609&623&614&616&338&348&341&347\\
970&964&973&963&249&247&254&240&666&660&669&659&825&823&830&816&426&420&429&419&9&7&14&0&618&612&621&611&345&343&350&336\\
972&962&971&965&255&241&248&246&668&658&667&661&831&817&824&822&428&418&427&421&15&1&8&6&620&610&619&613&351&337&344&342\\
967&969&960&974&244&250&243&253&663&665&656&670&820&826&819&829&423&425&416&430&4&10&3&13&615&617&608&622&340&346&339&349\\\hline
690&700&693&699&977&991&982&984&434&444&437&443&577&591&582&584&210&220&213&219&801&815&806&808&322&332&325&331&33&47&38&40\\
697&695&702&688&986&980&989&979&441&439&446&432&586&580&589&579&217&215&222&208&810&804&813&803&329&327&334&320&42&36&45&35\\
703&689&696&694&988&978&987&981&447&433&440&438&588&578&587&581&223&209&216&214&812&802&811&805&335&321&328&326&44&34&43&37\\
692&698&691&701&983&985&976&990&436&442&435&445&583&585&576&590&212&218&211&221&807&809&800&814&324&330&323&333&39&41&32&46\\
 \end{array}
\right).$}
\end{center}}
\noindent One can check that $C$ is   an MS$(32,2)$. \qed

In the following two sections, some families of  multimagic squares are provided to show the importance of  Construction \ref{conmscms}.

\section{ New families of bimagic squares}

In this section,  some new families of bimagic squares are provided by constructing a family of complementary magic squares.

 A {\it Kotzig array} of size $m\times n$, denoted by KA$(m,n)$, is an $m\times n$ array such that each row is a permutation of $\{0,1,2,\cdots,n-1\}$, each column has the same sum. Clearly, the common sum must be $\frac{1}{2}m(n-1)$, which is an integer only if $m(n-1)$ is even.

 Using  Kotzig array and diagonal orthogonal Latin squares we have the following.
 \begin{lem}\label{BCMS-11}
If there is a KA$(m,n)$ and  there is a pair of diagonal orthogonal  LS$(n)$s, then there is an $m$-SCMS$(n)$.
\end{lem}
 \begin{proof}
Let $K$ be an KA$(m, n)$, and let $A, B$ be a pair of  orthogonal diagonal LS($n$)s over $I_n$.
Let $$B_s=(b^{(s)}_{i,j}),\  b^{(s)}_{i,j}=k_{s,b_{i,j}},\ i, j\in I_{n},\  s\in I_m.$$
Then  $A, B_s$ are orthogonal diagonal LS($n$)s for each $s\in I_m$. Let
$$C_s=(c^{(s)}_{i,j}),  c^{(s)}_{i,j}=na_{i,j}+b^{(s)}_{i,j},i, j\in I_{n}\  s\in I_m.$$
It is easy to see that $C_s$ is an MS$(n)$ for each $s\in I_m$. We shall prove that $\{C_0,C_1,\dots,C_{m-1}\}$ is an $m$-SCMS($n$).

In fact, for any $i\in I_n$, we have
\begin{center}
$ \sum\limits_{s=0}^{m-1}\sum\limits_{j=0}^{n-1}(c_{i,j}^{(s)})^2
=\sum\limits_{s=0}^{m-1}\sum\limits_{j=0}^{n-1}(na_{i,j}+b^{(s)}_{i,j})^2\ \ \ \ \ \ \ \ \ \ \ \ \ \ \ \ \ \ \ \ \ \ \ \ \ \ \ \ $\\
$\ \ \ \ \ \ \ \ \ \ \ \ \ \ \ \ \ \ \ \ \ \ \ \ \ \ \ \ =n^2\sum\limits_{s=0}^{m-1} \sum\limits_{j=0}^{n-1}a_{i,j}^2
+\sum\limits_{s=0}^{m-1}\sum\limits_{j=0}^{n-1} (b_{i,j}^{(s)})^2
+2n\sum\limits_{j=0}^{n-1} a_{ij} \sum\limits_{s=0}^{m-1}b^{(s)}_{i,j}.$
\end{center}
By the definition of a KA$(m, n)$, we have
\begin{center}
 $\sum\limits_{s=0}^{m-1}b^{(s)}_{i,j}=\sum\limits_{s=0}^{m-1}k_{s,b_{i,j}}=m(n-1)/2, j\in I_n.$
\end{center}
Since $A,B_s$ are Latin squares, we have $\sum\limits_{j=0}^{n-1}a_{i,j}=n(n-1)/2$ and
\begin{center}
$\sum\limits_{j=0}^{n-1}(a_{i,j})^2=\sum\limits_{j=0}^{n-1}(b^{(s)}_{i,j})^2=n(n-1)(2n-1)/6, s\in I_m.$
\end{center}
Noting that     $S_2(n)=n(n^2-1)(2n^2-1)/6$, from above we have
\begin{center}
$ \sum\limits_{s=0}^{m-1}\sum\limits_{j=0}^{n-1}(c_{i,j}^{(s)})^2=m(n^2+1)n(n-1)(2n-1)/6+ n^2(n-1)m(n-1)/2$\\
$=mn(n^2-1)(2n^2-1)/6=mS_2(n).\ \ \ \ $
\end{center}
Similarly, we have
\begin{center}
$\sum\limits_{s=0}^{m-1}\sum\limits_{i=0}^{n-1}(c_{i,j}^{(s)})^2=mS_{2}(n),\ j\in I_n;\ \ \ \ \ \ \ \ \ \ \  \ \ \ \ \   \ $\\
\vskip5pt
$\sum\limits_{s=0}^{m-1}\sum\limits_{i=0}^{n-1}(c_{i,i}^{(s)})^2=\sum\limits_{s=0}^{m-1} \sum\limits_{i=0}^{n-1}({c^{(s)}_{i,n-1-i}})^2=mS_{2}(n).$
\end{center}
So, $\{C_0, C_1, \cdots, C_{m-1}\}$ is an $m$-SCMS($n$).
\end{proof}

%%%%%%%%%%%%%%%%%%%%%%%%%%%%%%
\noindent {\bf Example 4}
Let
{\renewcommand\arraystretch{0.7}
\setlength{\arraycolsep}{2.5pt}
\begin{center}
$A=${\scriptsize$\left(
        \begin{array}{ccccc}
1	&3	&0	&2	&4\\
2	&4	&1	&3	&0\\
3	&0	&2	&4	&1\\
4	&1	&3	&0	&2\\
0	&2	&4	&1	&3
        \end{array}
        \right),$}\ \ \ \
        $B=${\scriptsize$\left(
        \begin{array}{ccccc}
1&2&3&4&0\\
3&4&0&1&2\\
0&1&2&3&4\\
2&3&4&0&1\\
4&0&1&2&3
        \end{array}
        \right).$}
\end{center}}
\noindent Then $A, B$ is a pair of orthogonal diagonal LS$(5)$s.
 A KA($3,5$) $K$ is listed below.
{\renewcommand\arraystretch{0.7}
\setlength{\arraycolsep}{2.5pt}
\begin{center}
$K=${\scriptsize$\left(
        \begin{array}{ccccc}
4&1&3&0&2\\
0&1&2&3&4\\
2&4&1&3&0
        \end{array}
        \right).$}
\end{center}}
\noindent Let $B_s=(b^{(s)}_{i,j}),\  b^{(s)}_{i,j}=k_{s,b_{i,j}},\ i, j\in I_{5},\  s\in I_3,$ i.e.,
{\renewcommand\arraystretch{0.7}
\setlength{\arraycolsep}{2.5pt}
\begin{center}
$B_0=${\scriptsize$\left(
        \begin{array}{ccccc}
1&3&0&2&4\\
0&2&4&1&3\\
4&1&3&0&2\\
3&0&2&4&1\\
2&4&1&3&0
        \end{array}
        \right),$}\ \ \ \
        $B_1=${\scriptsize$\left(
        \begin{array}{ccccc}
1&2&3&4&0\\
3&4&0&1&2\\
0&1&2&3&4\\
2&3&4&0&1\\
4&0&1&2&3
        \end{array}
        \right),$}\ \ \ \
        $B_2=${\scriptsize$\left(
        \begin{array}{ccccc}
4&1&3&0&2\\
3&0&2&4&1\\
2&4&1&3&0\\
1&3&0&2&4\\
0&2&4&1&3
        \end{array}
        \right).$}
\end{center}}
\noindent Let $C_s=(c^{(s)}_{i,j}),  c^{(s)}_{i,j}=na_{i,j}+b^{(s)}_{i,j},i, j\in I_{5}, s\in I_3$, i.e.,
{\renewcommand\arraystretch{0.7}
\setlength{\arraycolsep}{2.5pt}
\begin{center}
$C_0=${\scriptsize$\left(
        \begin{array}{ccccc}
6&18&0&12&24\\
10&22&9&16&3\\
19&1&13&20&7\\
23&5&17&4&11\\
2&14&21&8&15
        \end{array}
        \right),$}\ \ \ \
$C_1=${\scriptsize$\left(
        \begin{array}{ccccc}
6&17&3&14&20\\
13&24&5&16&2\\
15&1&12&23&9\\
22&8&19&0&11\\
4&10&21&7&18
        \end{array}
        \right),$}\ \ \ \
        $C_2=${\scriptsize$\left(
        \begin{array}{ccccc}
9&16&3&10&22\\
13&20&7&19&1\\
17&4&11&23&5\\
21&8&15&2&14\\
0&12&24&6&18
        \end{array}
        \right).$}
\end{center}}
\noindent It is readily checked that $\{C_0, C_1, C_2\}$ is a 3-SCMS$(5)$. \qed

 For the existence of Kotzig array, Wallis \cite{W} proved the following.

 \begin{lem}\label{KKAA} \emph{(\cite{W})}
There is a  KA$(m,n)$ if and only if $m>1$ and $m(n-1)$ is even.
\end{lem}

For  the existence of orthogonal diagonal Latin squares,   Brown et al \cite{Brown} proved the following.

  \begin{lem}\label{ODLSS} \emph{(\cite{Brown})}
A pair of orthogonal diagonal LS($n$)s exists if and only if  $n\not\in\{2, 3, 6\}$.
 \end{lem}

By Lemma \ref{BCMS-11}, Lemma \ref{KKAA} and Lemma \ref{ODLSS} we have
\begin{lem}\label{BCMS}
There exists an $m$-SCMS$(n)$ whenever $m(n-1)$ is even and $m>1$,  and $n\not\in\{2,3,6\}$.
\end{lem}

By Construction \ref{conmscms} and Lemma \ref{BCMS}, we have the following.

 \begin{lem}\label{con-CMS}
 If there exists an MS$(m,2)$,  then there exists  an MS$(mn,2)$ whenever  $m(n-1)$ is even and $n\not\in\{2,3,6\}$.
 \end{lem}

 Let $\Omega_1=\{n| 8\leq n\leq 64\}$.  An MS$(n,2)$ for each  $n\in \Omega_1$ is listed in Boyer's website \cite{Boyer}.
Recently, Chen and Li \cite{Chen2} gave two families of bimagic squares. They showed that there exists an MS$(n,2)$ for each $n\in \Omega_2\cup\Omega_3$, where
$\Omega_2=\{n| n=n_1n_2, n_1\equiv n_2\pmod2, n_1,n_2\not\in \{2,3,6\} \}$ and
$\Omega_3=\{n|n\equiv0\pmod4, n\geq 4\}$.  These are summarized as follows.

\begin{lem} \emph{\cite{Boyer,Chen2}}  \label{64}
  There exists an MS$(n,2)$ for $n\in \Omega_1\cup \Omega_2\cup \Omega_3$.
 \end{lem}

Let $\Omega_4=\{mn | mn>64, 8\leq m\leq 64, m\equiv2\pmod4, n\geq5, n\equiv1\pmod2\}$, by using Lemma \ref{con-CMS} and Lemma  \ref{64}  we have
  \begin{thm}\label{main-ms2}
There exists an MS$(mn,2)$ for $mn\in \Omega_4$.
\end{thm}
 \begin{proof}
Let $mn\in \Omega_4$, there exists an MS$(m,2)$ by Lemma \ref{64}.  Since  $m$ is even, by Lemma \ref{con-CMS}, there exists an MS$(mn,2)$. \end{proof}

\noindent {\bf Remark } We should point out that  MS$(mn,2)$s given in Theorem \ref{main-ms2} are of orders $mn\equiv 2 \pmod{4}$, which are the  new ones comparing to those given in Lemma \ref{64}.

 %%%%%%%%%%%%%%%%%%%%%%%%%%%%%
\section{ New families of  trimagic squares}
In this section,  complementary  bimagic squares are considered  and some new families of trimagic squares are obtained.

\begin{lem}\label{t-5}
For even $t$, if there exists an MS$(n,t)$, then there exists  a $2l$-SCMS$(n,t)$ for $l\geq1$.
\end{lem}
\begin{proof}
 Suppose that $A$ is an  MS$(n,t)$ with even $t$.  Let $B=(n^2-1)J_{n\times n}-A$. It is readily checked that $B$ is also an MS$(n,t)$.  Note that  $t+1$ is odd,  for each $i\in I_n$, we have
   \begin{center}
$\sum\limits_{j=0}^{n-1}(a_{i,j}^{t+1}+b_{i,j}^{t+1})=\sum\limits_{j=0}^{n-1}(a_{i,j}^{t+1}+(n^2-1-a_{i,j})^{t+1})$\\
$\ \ \  \ \ \  \ \  \ \  \  \ \ \ \ \  \ \ \  \ \ \ \ \  \  \ \ =\sum\limits_{j=0}^{n-1}\sum\limits_{k=0}^{t}\binom{t+1}{k}(-1)^k (n^2-1)^{t+1-k}a_{i,j}^k.$\\
$\ \ \  \ \ \  \ \  \ \  \  \ \ \ \ \  \ \ \  \ \ \ \ \  \  \ \ =\sum\limits_{k=0}^{t}\binom{t+1}{k}(-1)^k (n^2-1)^{t+1-k}\sum\limits_{j=0}^{n-1}a_{i,j}^k.$\\
$\ \ \  \ \ \  \ \  \ \  \  \ \ \ \ \  \ \ \  \ \ \  \  \ \ =\sum\limits_{k=0}^{t}\binom{t+1}{k}(-1)^k (n^2-1)^{t+1-k}S_k(n),$
\end{center}
which is independent of $i$. Hence,
   \begin{center}
   $\sum\limits_{j=0}^{n-1}(a_{i,j}^{t+1}+b_{i,j}^{t+1})=\frac{1}{n} \sum\limits_{i=0}^{n-1}\sum\limits_{j=0}^{n-1}(a_{i,j}^{t+1}+b_{i,j}^{t+1})=\frac{1}{n}(\sum\limits_{i=0}^{n-1}\sum\limits_{j=0}^{n-1}a_{i,j}^{t+1}
   +\sum\limits_{i=0}^{n-1}\sum\limits_{j=0}^{n-1}b_{i,j}^{t+1})=2S_{t+1}(n).$
      \end{center}
Similarly, we have $\sum\limits_{i=0}^{n-1}(a_{i,j}^{t+1}+b_{i,j}^{t+1})=2S_{t+1}(n),j\in I_n,$ and
   \begin{center}
 $ \sum\limits_{i=0}^{n-1}(a_{i,i}^{t+1}+b_{i,i}^{t+1})=2S_{t+1}(n),\ \  \sum\limits_{i=0}^{n-1}(a_{i,n-1-i}^{t+1}+b_{i,n-1-i}^{t+1})=2S_{t+1}(n).$
 \end{center}
Thus, $\{A, B\}$ is a 2-SCMS$(n,t)$.
For $l\geq1$, let $A_i=A$, $B_i=B$, $i\in I_l$. It is readily checked that $\{A_0,B_0,A_1,B_1,\dots,A_{l-1},B_{l-1}\}$ is a $2l$-SCMS$(n,t)$.
\end{proof}

By Lemma \ref{t-5} and Lemma \ref{64}  we get the following.
 \begin{lem} Let  $l$ be a positive integer, then there exists a $2l$-SCMS$(n,2)$ for $n\in \Omega_1\cup \Omega_2\cup\Omega_3$.
\end{lem}

Some  trimagic squares are listed in the Boyer's website \cite{Boyer}.
   \begin{lem}\label{3-ms}  \emph{(\cite{Boyer})}
  There exists an MS($n, 3$) for all $m\in \{12,16,24,32,40,64,128\}$.
  \end{lem}

   \begin{lem}\label{qm2}
  There exists an MS($2^m, 3$) for all $m\geq4$.
  \end{lem}
    \begin{proof}
  There exists an MS($2^m, 3$) for $m\in \{4,5,6,7\}$ by Lemma \ref{3-ms}. For any $m\ge 8$, we can write $m=4n_1+5n_2+6n_3+7n_4$, where $n_1,n_2,n_3,n_4$ are nonnegative integers. By the product construction of multimagic squares mentioned in Section 1 we get  an MS($2^m, 3$).
\end{proof}

 \begin{lem}\label{evenmoddn}
For even $m$ and odd $t$, if there exists an MS$(m,t)$ and there exists an MS$(n,t-1)$, then there exists an MS$(mn,t)$.
  \end{lem}
  \begin{proof}
  Since there exists an   MS$(n,t-1)$, $m$ is even and $t$ is odd,   there exists an   $m$-SCMS$(n,t-1)$ by Lemma \ref{t-5}.
 Hence, there exists an MS$(mn,t)$ by Construction \ref{conmscms}.
\end{proof}

\begin{thm}\label{main-ms3}
There exists an MS$(mn,3)$ for any $m\in M=\{12,24,40\}\cup\{2^l|l\geq4\}$ and $n\in \Omega_1\cup \Omega_2 \cup \Omega_3 \cup \Omega_4$.
\end{thm}
\begin{proof}
For any $m\in M=\{12,24,40\}\cup\{2^l|l\geq4\}$, there exists an MS$(m,3)$ by Lemma \ref{3-ms} and Lemma \ref{qm2}. For any $n\in \Omega_1\cup \Omega_2 \cup \Omega_3 \cup \Omega_4$,  there exists an MS$(n,2)$  by Lemma \ref{64} and Theorem \ref{main-ms2}. So there exists an MS$(mn,3)$ by Lemma \ref{evenmoddn}.
\end{proof}

 %\vskip10pt

% \noindent\textbf{Acknowledgements} The authors would like to thank Professor L. Zhu of Suzhou University for helpful discussions.
%The authors would like to thank Professor J. Lei of Hebei Normal  University for useful advices.


\begin{thebibliography}{99}
{\small \bibitem{Abe}  G. Abe, Unsolved problems on magic squares, Discrete Math. 127 (1994) 3-13.

 \bibitem{Ahmed}  M. Ahmed, Algebraic combinatorics of magic squares, Ph.D. Dissertation, University of California, Davis, 2004.

\bibitem{Andrews} W. S. Andrews, Magic Squares and Cubes, 2nd ed., Dover, New York, 1960.

\bibitem{Boyer} C. Boyer, Multimagic squares,    www.multimagie.com.

\bibitem{Brown}  J. W. Brown, F. Cherry, L. Most, M. Most, E. T. Parker,   W. D.Wallis, The spectrum of orthogonal diagonal latin squares,
in Graphs, Matrices and Designs, R. S. Rees, ed., Dekker, New York, 1993, 43-49.

 \bibitem{Cammann} S. v. R, Cammann,   Magic squares, in Encyclop{\ae}dia Britannica, 14th Ed.,  Chicago, 1973.

 \bibitem{Chen2}	K. Chen, W. Li, Existence of normal bimagic squares,  Discrete Math.   312(2012)  3077-3086.

\bibitem{Handbook} C. J. Colbourn, J. H. Dinitz (eds.),  Handbook of combinatorial designs, 2nd Edition,  Chapman and Hall/CRC, Boca Raton FL, 2007.

\bibitem{Derksen} H. Derksen, C. Eggermont, A. van den Essen, Multimagic squares, Amer. Math. Monthly 114 (2007) 703-713.

\bibitem{Gergely} E. Gergely, A simple method for constructing doubly diagonalized latin squares, J. Combin. Theory A  16 (1974) 266-272.

 \bibitem{Kim} Y. Kim, J. Yoo, An algorithm for constructing magic squares, Discrete Appl. Math. 156 (2008) 2804-2809.


\bibitem{Li} W. Li, D. Wu, F. Pan, A construction for doubly pandiagonal magic squares, Discrete Math. 312(2012) 479-485.

\bibitem{Nordgren} R. P. Nordgren, On properties of special magic square matrices, Linear Algebra Appl.   437 (2012)  2009-2025.

\bibitem{PF} G. Pfeffermann,  Les Tablettes du Chercheur, Journal des Jeux d'Esprit et de Combinaisons, (fortnightly magazine) issues of 1891 Paris.

  \bibitem{W} W. D. Wallis, Vertex magic labelings of multiple graphs,   Congr. Numer. 152(2001) 81-83.

\bibitem{Zhang}	Y. Zhang,   K. Chen,   J. Lei, Large sets of orthogonal arrays and multimagic squares,  J. Combin. Des.   21(2013) 390-403.

\bibitem{Zhang2} Y. Zhang, J. Lei, Multimagic rectangles based on large sets of orthogonal arrays, Discrete Math. 313 (2013)  1823-1831.
}
\end{thebibliography}
\end{document}